\documentclass[12pt]{amsart} 
\usepackage{amssymb}
\usepackage{latexsym}
\usepackage{amsmath}
\usepackage{amsthm}
\usepackage{amsfonts}

\usepackage[usenames,dvipsnames]{xcolor}

\newtheorem{theorem}{Theorem}
\newtheorem{lemma}[theorem]{Lemma}
\newtheorem{corollary}[theorem]{Corollary}
\newtheorem{prop}[theorem]{Proposition}

\theoremstyle{definition}

\newtheorem*{acknowledgements*}{Acknowledgements}
\theoremstyle{remark}
\newtheorem{remark}[theorem]{Remark}
\newtheorem{note}[theorem]{Note}

\usepackage{color}

\newcommand{\T}{\mathbb T}

\newcommand{\D}{\mathbb D}
\def\he{{H}(X)}
\def\hf{H(\Lambda(\varphi))}
\def\laf{\Lambda(\varphi)}
\def\h{H^1(\D)}

\newcounter{obs}

\begin{document}

\title
[On the set of extreme points of the unit ball]
{On the set of extreme points of the unit ball\\ of a Hardy-Lorentz space}

\author{Sergey V. Astashkin}
\address{Astashkin: Department of Mathematics, Samara National Research University, Moskovskoye shosse 34, Samara, 443086, Russian Federation; Saint Petersburg University, 7/9 Universitetskaya nab., Saint Petersburg, 199034, Russian Federation; Bahcesehir University, 34353, Istanbul, Turkey}
\email{astash56@mail.ru}
\urladdr{www.mathnet.ru/rus/person/8713}
\thanks{$^\dagger$\,This work was performed at the Saint Petersburg Leonhard Euler International Mathematical Institute and supported by the Ministry of Science and Higher Education of the Russian Federation (agreement no. 075-15-2022-287}).

\maketitle

\date{\today}


\begin{abstract}
We prove that every measurable function $f:\,[0,a]\to\mathbb{C}$ such that $|f|=1$ a.e. on $[0,a]$ is an extreme point of the unit ball of the Lorentz space $\Lambda(\varphi)$ on $[0,a]$ whenever $\varphi$ is a not linear, strictly increasing, concave, continuous function on $[0,a]$ with $\varphi(0)=0$.  As a consequence, we complement the classical de Leeuw-Rudin theorem on a description of extreme points of the unit ball of $H^1$ showing that $H^1$ is a unique Hardy-Lorentz space $H(\Lambda(\varphi))$, for which every extreme point of the unit ball is a normed outer function. Moreover, assuming that $\varphi$ is strictly increasing and strictly  concave, we prove that every function $f\in H(\Lambda(\varphi))$, $\|f\|_{H(\Lambda(\varphi))}=1$, such that the absolute value of its nontangential limit  ${f}(e^{it})$ is a constant on some set of positive measure of $[0,2\pi]$, is an extreme point of the unit ball of $H(\Lambda(\varphi))$.
\end{abstract}

Primary classification: 46E30

Secondary classification(s): 30H10, 30J05, 46A55, 46B22 


\maketitle


\section{Introduction}
\label{S0}


In 1958, de Leeuw and Rudin gave the following remarkable characterization of extreme points of the unit ball
of the Hardy space $H^1:=\h$ (see \cite[Theorem 1]{deleeuw-rudin} or \cite[Chapter~9]{hoffman}).

\begin{theorem}[de Leeuw-Rudin]\label{t-1}
A function in $H^1$ is an extreme point of the unit ball of $H^1$ if 
and only if it has norm one and is an outer function.
\end{theorem}

In 1974, this result has been partially extended by Bryskin and Sedaev  to general Hardy-type spaces $H(X)$ such that the norm $\|\cdot\|_X$ of a symmetric space $X$ on $[0,2\pi]$ is strictly monotone \cite[Theorem 1]{bryskin-sedaev}.
%
%

The norm $\|\cdot\|_X$ of a symmetric space $X=X[0,a]$ is 
\textit{strictly monotone} whenever from the conditions   
$x,y\in X$, $|x|\le|y|$ a.e. on $[0,a]$ and
\begin{equation*}\label{1}
m\{t\in [0,a]:\,|x(t)|<|y(t)|\}>0
\end{equation*}
it follows that $\|x\|_X<\|y\|_X$.

%
%

\begin{theorem}[Bryskin-Sedaev]\label{t-2}
If  the norm of a symmetric space $X$ is strictly monotone,  then every normed outer function in $\he$ is an extreme point of the unit ball of ${\he}$. 
\end{theorem}

In this paper, we will be interested mainly in Lorentz and Hardy-Lorentz spaces. Let $\varphi$ be an increasing, concave, continuous, real-valued function on $[0,a]$ such that $\varphi(0)=0$ and $\varphi(a)=1$.
The {\it Lorentz space} $\Lambda(\varphi):=\Lambda(\varphi)[0,a]$ consists of all complex-valued functions $x(t)$ measurable on $[0,a]$ and satisfying the condition:
\begin{equation}
\label{eqLor} 
\|x\|_{\Lambda(\varphi)}:=\int_0^{a} x^{*}(t)\, 
d\varphi(t)<\infty,
\end{equation}
where $x^*(t)$ is the decreasing rearrangement of $|x|$. Since $\varphi$ being strictly monotone implies that  the norm in $\laf$ is strictly monotone (see Lemma \ref{0} below), the result of Theorem \ref{t-2} holds, in particular, under the above condition imposed on $\varphi$, for the space ${\hf}$ (see \cite[Corollary]{bryskin-sedaev}).




In 1978, E.~Semenov stated the problem of describing the set of extreme points of the unit ball of a Hardy-Lorentz space $\hf$, \cite{semenov} as a part of a collection of 99 problems in linear and complex analysis
from the  Leningrad branch of the Steklov Mathematical Institute,  \cite[5.1 pp.23-24]{khavin-1978}. Later on, the same problem appeared in the subsequent list of problems in 1984, \cite[1.6 pp.22-23]{khavin-1984}, and in 1994, \cite[1.5 p.12]{khavin-1994}. 

Observe that the mapping 
\begin{equation}
\label{map} 
f\mapsto \tilde{f}:=f(e^{it}),
\end{equation}
where we denote by $f(e^{it})$ the nontangential limit of $f(z)$ as $z\to e^{it}$, is an isometric inclusion of a Hardy-type space $H(X)$ into the underlying symmetric space $X=X[0,2\pi]$.
Therefore, in connection with the above problem, it is important to find out which functions are being extreme points of the unit ball of a Lorentz space.
The following result (see \cite[Proposition 2.2]{carothers-dilworth-trautman}, \cite[Theorem 1]{carothers-turett} and \cite{semenov}) contains a description of extreme points for the unit ball of  $\Lambda(\varphi)$ whenever $\varphi$ is strictly concave.

\begin{theorem}
\label{t-0}
If $f$ is an extreme point of the unit ball of the space $\Lambda(\varphi):=\Lambda(\varphi)[0,a]$,  then 
\begin{equation}
\label{1a}
f(t)=\frac{\varepsilon(t)}{\varphi(m(E))}\chi_E(t),\;\;0\le t\le a,
\end{equation}
where $E\subset [0,a]$ is a measurable set 
and $\varepsilon$ is a complex-valued measurable function such that $|\varepsilon|=1$ a.e. 

Moreover, if the function $\varphi$ is strictly concave, $f$ is an extreme point of the unit ball of $\Lambda(\varphi)$ if and only if $f$ can be represented in the form \eqref{1a}.
\end{theorem}

Our first main result refines Theorem \ref{t-0}
in the case when the function $\varphi$, in general, fails to be strictly concave, showing that every measurable function $f:\,[0,2\pi]\to\mathbb{C}$ such that $|f|=1$ a.e. on $[0,2\pi]$ is an extreme point of the unit ball of the space $\Lambda(\varphi)$ if $\varphi$ is not linear and strictly increasing. This result
can be compared with the easy fact that the set of extreme points of the unit ball of the space $L^1$ (which coincides with $\Lambda(\varphi_0)$, where $\varphi_0(t)=t/(2\pi)$), is empty. As a consequence, we get that $H^1:={\h}$ is a unique space in the class of Hardy-Lorentz spaces ${\hf}$, where $\varphi$ is a strictly increasing concave function on $[0,2\pi]$, for which every extreme point of the unit ball is a normed outer function.

In the recent note \cite{CC}, it was proved that every inner  function is an extreme point of the unit ball of ${\hf}$ whenever $\varphi$ is  strictly increasing and strictly  concave on $[0,2\pi]$. 
However, one can easily see that this result is an immediate consequence of Theorem \ref{t-0}.  Indeed, for each inner function $f$ we certainly have  $|\tilde{f}|=1$ a.e. on $[0,2\pi]$. Therefore, $\tilde{f}$ is an extreme point of the unit ball of the space $\Lambda(\varphi)$. Since mapping \eqref{map} 
includes ${\hf}$ into $\Lambda(\varphi)$ as an isometric subspace, it follows then that $f$ is an extreme point of the unit ball of the space ${\hf}$.

The second result of the paper resolves the above Semenov's problem for the class of functions $f\in H(\Lambda(\varphi))$ such that $|\tilde{f}|$ is constant on some subset of positive measure of $[0,2\pi]$. More precisely, we prove that such a normed function is an extreme point of the unit ball of ${\hf}$ if $\varphi$ is strictly increasing and strictly  concave on $[0,2\pi]$. Note that this class contains, in particular,  functions which are being neither inner, nor outer ones. 

In a subsequent paper\footnote{S.V. Astashkin, {\it On the set of extreme points of the unit ball of a Hardy-Lorentz space, 2} (manuscript).}, we investigate a different case when $f$ is a normed function in $H(\Lambda(\varphi))$ such that $m\{t:\,|f(e^{it})|=c\}=0$ for each $c\ge 0$ and $f(a)=0$ for some point $a\in\mathbb{D}$.


It is worth also to say briefly about resent investigations of the structure of extreme points of the unit ball for other types of function spaces which can be treated as a generalization of the classical $H^1$. We mention here only the paper \cite{dyak}, where a similar study was carried out for a punctured Hardy space $H^1_K$ consisting of the integrable functions $f$ on the unit circle whose Fourier coefficients $c_k(f)$ vanish either $k<0$ or $k\in K$ ($K$ is a fixed finite set of positive integers). 

The author thanks Professors A. Baranov, G. Curbera and I. Kayumov for useful discussions related to matters considered in the paper. 



\section{Preliminaries.}
\label{S1}


\subsection{Symmetric spaces.}

A detailed exposition of the theory of symmetric spaces see in \cite{bennett-sharpley,krein-petunin-semenov,LT2}.

Let $a>0$. A Banach space $X:=X[0,a]$ of complex-valued Lebesgue-measurable functions on
the measure space $([0,a],m)$, where $m$ is the Lebesgue measure on the interval $[0,a]$, is called {\it symmetric} (or {\it rearrangement invariant}) if from the conditions $y \in X$ and
$x^*(t)\le y^*(t)$ almost everywhere (a.e.) on $[0,a]$ it follows that $x\in X$ and ${\|x\|}_X \le {\|y\|}_X.$ Here, $x^*(t)$ is the right-continuous nonincreasing {\it rearrangement} of $|x|$, i.e., 
$$
x^{*}(t):=\inf \{ \tau\ge 0:\,m\{s\in [0,a]:\,|x(s)|>\tau\}\le t \},\;\;0<t\le a.$$
Measurable functions  $x$ and $y$ on $[0,a]$ are said to be {\it equimeasurable} whenever 
$$
m\{s\in [0,a]:\,|x(s)|>\tau\}=m\{s\in [0,a]:\,|y(s)|>\tau\},\;\;\tau>0.$$
In particular, a complex-valued measurable function $x(t)$ and its rearrangement $x^*(t)$ are  equimeasurable. Clearly, if $X$ is a symmetric space, $y \in X$ and $x$ and $y$ are equimeasurable on $[0,a]$, then $x\in X$ and $\|x\|_X=\|y\|_X.$ 

Every symmetric space $X$ on $[0,a]$ satisfies  the embeddings $L^\infty[0,a]{\subset} X{\subset} L^1[0,a]$.


The family of symmetric spaces includes many classical spaces 
appearing in analysis, in particular, $L^p$-spaces with $1\le p\le\infty$, which are defined in  the usual way, Orlicz, Lorentz, Marcinkiewicz spaces and many others. Next, we will be interested primarily in Lorentz spaces.

In what follows, $\varphi$ is an increasing, concave,  continuous function on $[0,a]$ such that $\varphi(0)=0$ and $\varphi(a)=1$.
Then (see Section \ref{S0}), the {\it Lorentz space} $\Lambda(\varphi):=\Lambda(\varphi)[0,a]$ is equipped with norm  \eqref{eqLor}.
In particular, if $\varphi(t)=a^{-1/p}t^{1/p}$, $1<p<\infty$, this space is denoted, as usual, by $L^{p,1}$ (see, e.g., \cite[\S\,4.4]{bennett-sharpley} or \cite[\S\,II.6.8]{krein-petunin-semenov}).  

For every $\varphi$, $\Lambda(\varphi)$ is a separable symmetric space. Moreover, such a space has the Fatou property, that is, the conditions $x_n\in \Lambda(\varphi)$, $n=1,2,\dots,$ $\sup_{n=1,2,\dots}\|x_n\|_{\Lambda(\varphi)}<\infty$, and $x_n\to{x}$ a.e. on $[0,a]$ imply that $x\in \Lambda(\varphi)$ and $\|x\|_{\Lambda(\varphi)}\le \liminf_{n\to\infty}\|x_n\|_{\Lambda(\varphi)}.$
 
\subsection{Hardy-type spaces of analytic functions.}

Let $X$ be a symmetric space with the Fatou property on the unit circle
$$
\T:=\{e^{it}:\, 0\le t\le 2\pi\},$$ 
which will be identified further with the interval $[0,2\pi]$. The Hardy-type space $\he$ consists of all $f\colon\D\to\mathbb{C}$, where $\D:=\{z\in \mathbb{C}:\,|z|<1\}$, such that
\begin{equation*}
\|f\|_{\he}:=\sup_{0\le r<1}\|f_r\|_X<\infty,
\end{equation*}
where $f_r(e^{it}):=f(re^{it})$ if $0\le r<1$ and $t\in[0,2\pi]$ (see \cite{bryskin-sedaev}). In particular, if $X=L^p$, $1\le p\le\infty$, we get the classical $H^p$-spaces, $H^p=H(L^p)$  \cite{duren,hoffman,koosis}.  

Since $X\subset L^1$ for every symmetric space $X$, we have $\he\subseteq H^1$. Consequently, for each function $f\in\he$ and almost all $t\in[0,2\pi]$ there exists the nontangential limit $f(e^{it}):=\lim_{z\to e^{it}}f(z)$. Moreover, precisely as in the classical $H^p$ case, the mapping \eqref{map} 
is an isometry between $\he$ and the subspace of $X$, consisting of all functions $g\in X$ such that the Fourier coefficients 
$$
c_k(g):=\frac{1}{2\pi}\int_0^{2\pi} g(t)e^{-ikt}\,dt
$$
vanish if $k<0$.



Recall the canonical (inner-outer) factorization theorem for $H^1$ functions. A function $I\in H^\infty$ is called {\it inner} if 
$|\tilde{I}|=1$ a.e. on $[0,2\pi]$. Also, a non-null function $F\in H^1$ is termed {\it outer} if
$$
\ln |F(0)|=\frac{1}{2\pi}\int_{0}^{2\pi}\ln |\tilde{F}(s)|\,ds.$$
As is well known, every function $f\in H^1$, $f\not\equiv 0$, can be represented as $f=IF$, where $I$ is inner and $F$ is outer (see, for  instance, \cite[\S\,IV.D.$4^0$]{koosis}).

Let $g\in H(X)$, where $X$ is a symmetric space on $[0,2\pi]$. Since  $g\in H^1$, we have
$$
\int_{0}^{2\pi}|\ln|\tilde{g}(t)||\,dt<\infty$$
(see, for instance, \cite[$\S\,IV.D.4^0$]{koosis}). Conversely, if $\mu\in X$, $\mu\ge 0$, and $\ln\mu\in L^1$, then there are functions $g\in H(X)$ such that 
\begin{equation}
\label{1a-dop}
|\tilde{g}|=\mu.
\end{equation}
Indeed, for example, the outer function
$$
F(z):=\exp\left\{\frac{1}{2\pi}\int_0^{2\pi}\frac{e^{it}+z}{e^{it}-z}\ln\mu(t)\,dt+i\lambda\right\},\;\;z\in\mathbb{D},$$
where $\lambda\in\mathbb{R}$, satisfies condition \eqref{1a-dop} \cite[$\S\,IV.D.4^0$]{koosis} and hence belongs to $H(X)$.

\subsection{Extreme points and some related definitions.}

Given Banach space $X=(X, \|\cdot\|)$, denote by ${\rm ball}\,(X)$ the closed unit ball of $X$, that is,
$$
{\rm ball}\,(X):=\{x\in X:\,\|x\|\le 1\}.$$
A point $x_0\in {\rm ball}\,(X)$ is said to be {\it extreme} for the set ${\rm ball}\,(X)$ if $x_0$ is not the midpoint of any nondegenerate segment contained in ${\rm ball}\,(X)$.  

Clearly, every extreme point $x_0$ of ${\rm ball}\,(X)$ is a normed element, i.e., $\|x_0\|=1$. Moreover, $x_0$ is an extreme point of ${\rm ball}\,(X)$ if and only if from $y\in X$ and $\|x\pm y\|_X\le 1$ it follows that $y=0$. 



We say that a function $\varphi:\, [0,a]\to \mathbb{R}$ is {\it increasing} (resp. {\it strictly increasing}) if from $0\le t_1<t_2\le a$ it follows $\varphi(t_1)\le \varphi(t_2)$ (resp. $\varphi(t_1)< \varphi(t_2)$). A (strictly) decreasing function is defined similarly.

A real-valued function $\varphi$ defined  on $[0,a]$ is  called {\it strictly concave} if 
$$
\varphi((1-\alpha)t_1+\alpha t_2)<(1-\alpha)\varphi(t_1)+\alpha \varphi(t_2)$$
for any $\alpha\in (0,1)$ and $0\le t_1,t_2\le a$, $t_1\ne t_2$.

\section{Auxiliary results}
\label{S2}


The proof of the following assertion is immediate and hence we skip it.

\begin{lemma}
\label{0}
If $\varphi$ is a strictly increasing, concave function on $[0,a]$,  $\varphi(0)=0$, $\varphi(a)=1$, then the norm of the Lorentz space $\Lambda(\varphi)$ is strictly monotone. 
\end{lemma}

An inspection of the proof of de Leeuw-Rudin theorem in \cite[Theorem 1]{deleeuw-rudin} or \cite[Chapter~9]{hoffman} implies the following useful result (see also Lemma 1 in \cite{bryskin-sedaev}). 

\begin{lemma}
\label{l0}
Suppose the norm of a symmetric space $X$ on $[0,a]$ is strictly monotone. If $\|f\|_X=1$, $g\in X$, $m\{t:\,f(t)=0\}=0$, and 
$\|f\pm g\|_X\le 1$, then ${g}=h{f}$, where $h$ is a real-valued function, $|h|\le 1$ a.e. on $[0,a]$.
\end{lemma}

A proof of the following result can be found in \cite[Lemma 6]{CC} (for the special case $\varphi(t)=a^{-1/p}t^{1/p}$, $1<p<\infty$, see also   \cite[Lemma 1]{carothers-turett}).

\begin{lemma}\label{l-3}
Let $\varphi$ be a strictly increasing and  strictly concave  
function on $[0,a]$ such that $\varphi(0)=0$, $\varphi(a)=1$. Assuming that  $f,g \in \Lambda(\varphi)$ satisfy 
$$\|f\|_{\Lambda(\varphi)} + \|g\|_{\Lambda(\varphi)}=\|f+g\|_{\Lambda(\varphi)},
$$
we have 
$$ f^* + g^*=(f+g)^*\;\;\mbox{a.e. on}\;[0,a]. 
$$
\end{lemma}

As is known (see, for instance, \cite[\S\,II.2, the proof of Property $7^0$]{krein-petunin-semenov}, for every function $x\in L_1[0,a]$ there exists a family of measurable subsets $\{E_t(x)\}_{0<t\le a}$ of the interval $[0,a]$ such that $m(E_t(x))=t$, $0<t\le a$, $E_{t_1}(x)\subset E_{t_2}(x)$ if $0<t_1<t_2\le a$, the embeddings   
\begin{equation}
\label{pra1}
\{s\in [0,a]:\,|x(s)|>x^*(t)\}\subset E_t(x)\subset \{s\in [0,a]:\,|x(s)|\ge x^*(t)\},\;\;0<t\le a,
\end{equation}
hold, and
\begin{equation}
\label{pra2}
\int_{E_t(x)}|x(s)|\,ds=\int_0^t x^*(s)\,ds,\;\;0<t\le a.
\end{equation}

\begin{prop}
\label{propos}
Let $\varphi$ be a strictly increasing and  strictly concave  
function on $[0,a]$, $\varphi(0)=0$, $\varphi(a)=1$. Suppose  $u,v \in \Lambda(\varphi)[0,a]$, $u$ is nonnegative, $v$  is real-valued, $\|u\|_{\Lambda(\varphi)}=1$ and $\|u\pm v\|_{\Lambda(\varphi)}\le 1$. 

Then both functions $u(s)\pm v(s)$  are nonnegative a.e. on the interval $[0,a]$ and there exists a family of measurable subsets $\{E_t\}_{0<t\le a}$ of $[0,a]$ which satisfies the following properties: $m(E_t)=t$, $0<t\le a$, $E_{t_1}\subset E_{t_2}$, $0<t_1<t_2\le a$,  
\begin{equation}
\label{pr1}
\int_{E_t}u(s)\,ds=\int_0^t u^*(s)\,ds,\;\;0<t\le a,
\end{equation}
and
\begin{equation}
\label{pr2-e}
\int_{E_t}(u(s)\pm v(s))\,ds=\int_0^t (u\pm v)^*(s)\,ds,\;\;0<t\le a.
\end{equation}
\end{prop}
\begin{proof}
First, from the hypothesis of the proposition it follows that $\|u\pm v\|_{\Lambda(\varphi)}= 1$. Hence, 
$$
\|u+v\|_{\Lambda(\varphi)}+\|u-v\|_{\Lambda(\varphi)}=2=\|2u\|_{\Lambda(\varphi)}.$$
Thus, the functions $u+v$ and $u-v$ satisfy the conditions of Lemma  \ref{l-3}, and consequently
\begin{equation}
\label{pr3}
(u+v)^*+(u-v)^*=2u^*\;\;\mbox{a.e. on}\;[0,a].
\end{equation}
In view of \cite[\S\,II.2, Property $9^0$]{krein-petunin-semenov}, the last inequality is equivalent to the fact that the functions $u\pm v$ are of the same sign a.e. on $[0,a]$ and as families of subsets $\{E_t(u+v)\}_{0<t\le a}$ and $\{E_t(u-v)\}_{0<t\le a}$ (see the  discussion before the proposition),
can be chosen the same family $\{E_t\}_{0<t\le a}$. Since $u\ge 0$, this implies, in particular, that the both functions $u\pm v$ are a.e. nonnegative. Therefore, by \eqref{pra2}, we obtain \eqref{pr2-e}, that is, we have
\begin{equation*}
\int_{E_t}(u(s)+v(s))\,ds=\int_0^t (u+ v)^*(s)\,ds,\;\;0<t\le a,
\end{equation*}
and
\begin{equation*}
\int_{E_t}(u(s)-v(s))\,ds=\int_0^t (u- v)^*(s)\,ds,\;\;0<t\le a.
\end{equation*}
Summing up these equations and applying \eqref{pr3}, we get   
\begin{eqnarray*}
2\int_{E_t}u(s)\,ds&=&
\int_0^t ((u+v)^*(s)+(u-v)^*(s))\,ds\\&=&
2\int_0^t u^*(s)\,ds,\;\;0<t\le a,
\end{eqnarray*}
which implies \eqref{pr1}. This completes the proof.
\end{proof}

\begin{lemma}
\label{prop1a}
Let $a>0$ and let $\varphi$ be a nonlinear, strictly increasing, concave, continuous function on $[0,a]$ such that $\varphi(0)=0$ and $\varphi(a)=1$. Assuming that a function $\rho:\,[0,a]\to\mathbb{R}$ satisfies the condition: $0<m({\rm supp}\,\rho)<a$, we have  
$$
\int_0^a\rho^*(t)\,d\varphi(t)>-\int_0^a\rho^*(t)\,d\varphi(a-t).$$
\end{lemma}
\begin{proof}
First, we put $b:=m({\rm supp}\,\rho)$. Since the function $\varphi$ is concave, the derivative $\varphi^\prime(t)$ exists a.e. on $[0,a]$ and decreases. Moreover, by the assumption, $\varphi$ is not linear, and hence $\varphi^\prime(t)$ is not constant on $[0,a]$. Therefore, since $b<a$, there exists $0<\delta<a-b$ such that for all $0<t<\delta$ we have
$$
\varphi^\prime(a-t)<\varphi^\prime(t).$$
Because $\rho\ne 0$, this implies that
\begin{equation}
\label{fist ine}
\int_0^\delta \rho^*(t)\varphi^\prime(t)\,dt>\int_0^\delta \rho^*(t)\varphi^\prime(a-t)\,dt.
\end{equation}

Next, from the inequality $a>b+\delta$ it follows
$$
\varphi^\prime(a+t-b-\delta)\le \varphi^\prime(t),\;\;\delta< t<b.$$
Moreover, the increasing function $t\mapsto \varphi^\prime(a-t)$ and decreasing function $t\mapsto \varphi^\prime(a+t-b-\delta)$ are equimeasurable, when restricted to the interval $[\delta,b]$.
Thus, in view of the equality $\rho'(t)=0$, $t\in(b,a]$, and a well-known property of the decreasing rearrangement (see, for instance, \cite[\S\,II.2, Property $16^0$]{krein-petunin-semenov}), we have
\begin{eqnarray*}
\int_\delta^a \rho^*(t)\varphi^\prime(t)\,dt &=& \int_\delta^b \rho^*(t)\varphi^\prime(t)\,dt\\ &\ge& \int_\delta^b \rho^*(t)\varphi^\prime(a+t-b-\delta)\,dt\\ &\ge& \int_\delta^b \rho^*(t)\varphi^\prime(a-t)\,dt\\ &=& \int_\delta^a \rho^*(t)\varphi^\prime(a-t)\,dt.
\end{eqnarray*}
Combining this together with \eqref{fist ine}, we obtain
\begin{eqnarray*}
\int_0^a \rho^*(t)\varphi^\prime(t)\,dt &=& \int_0^\delta \rho^*(t)\varphi^\prime(t)\,dt+\int_\delta^a \rho^*(t)\varphi^\prime(t)\,dt\\ &>& \int_0^\delta \rho^*(t)\varphi^\prime(a-t)\,dt+\int_\delta^a \rho^*(t)\varphi^\prime(a-t)\,dt\\ &=&\int_0^a \rho^*(t)\varphi^\prime(a-t)\,dt=-
\int_0^a \rho^*(t)\,d\varphi(a-t),
\end{eqnarray*}
and thus everything is done.
\end{proof}

\begin{remark}
The simple example $\rho(t)=1$, $0\le t\le a$ shows that the assumption $m({\rm supp}\,\rho)<a$ in the last lemma cannot be skipped.
\end{remark}

\section{Main results}
\label{S3}


Our first result refines Theorem \ref{t-0} in the case when a concave function $\varphi$, in general, fails to be strictly concave. As a consequence, we get that the classical $H^1$ is a unique space in the class of Hardy-Lorentz spaces ${\hf}$, with strictly increasing, concave $\varphi$, such that each extreme point of the unit ball of ${\hf}$ is a normed outer function.
 

\begin{theorem}\label{Th1a}
Suppose $\varphi$ is a nonlinear, strictly increasing, concave, continuous function on $[0,a]$ such that $\varphi(0)=0$, $\varphi(a)=1$.
Then, every measurable function $f:\,[0,a]\to\mathbb{C}$ such that $|f|=1$ a.e. on $[0,a]$ is an extreme point of the unit ball of the space $\Lambda(\varphi)$.
\end{theorem}

\begin{proof}
Assuming that $f$ is not an extreme point of the unit ball of the space $\Lambda(\varphi)$, we find a function $g\in \Lambda(\varphi)$, $g\ne 0$, such that
\begin{equation}\label{equ1a}
\|f\pm g\|_{\Lambda(\varphi)}=1.
\end{equation}
Since $\varphi$ is strictly increasing, $m\{t:\,f(t)=0\}=0$ and $\|f\|_{\Lambda(\varphi)}=1$, by Lemmas \ref{0} and \ref{l0}, we conclude that $g=h\cdot f$, where $h$ is a real-valued function, $|h|\le 1$ a.e. on $[0,a]$. Thus, from the assumption $|f|=1$ a.e. on $[0,a]$ and \eqref{equ1a} it follows that 
\begin{equation}\label{equ2a}
\|1\pm h\|_{\Lambda(\varphi)}=1.
\end{equation}

Setting
$$
E_+:=\{t\in [0,a]:\,h(t)>0\},\;\;E_-:=\{t\in [0,a]:\,h(t)<0\},$$
by the definition of the norm in a Lorentz space, we have
\begin{eqnarray}
\|1+h\|_{\Lambda(\varphi)}&=& \int_0^{m(E_+)}(1+h)^*(t)\,d\varphi(t)+\int_{a-m(E_-)}^{a}(1+h)^*(t)\,d\varphi(t)\nonumber\\ &+& \varphi(a-m(E_-))-\varphi(m(E_+)).
\label{equa base}
\end{eqnarray}

A direct calculation shows that 
$$
(1+h)^*(t)=(h\chi_{E_+})^*(t)+1,\;\;0\le t\le m(E_+).$$
and
$$
(1+h)^*(t)=1-(h\chi_{E_-})^*(a-t),\;\;a-m(E_-)\le t\le a.$$
Therefore,
\begin{eqnarray}
\|1+h\|_{\Lambda(\varphi)}&=& \varphi(m(E_+))+\int_0^{m(E_+)}(h\chi_{E_+})^*(t)\,d\varphi(t)\nonumber\\ &+& \int_{a-m(E_-)}^{a}(1-(h\chi_{E_-})^*(a-t))\,d\varphi(t)\nonumber\\ &+& \varphi(a-m(E_-))-\varphi(m(E_+))\nonumber\\&=& \int_0^{m(E_+)}(h\chi_{E_+})^*(t)\,d\varphi(t) + 1-\varphi(a-m(E_-))\nonumber\\ &-&\int_{a-m(E_-)}^{a}(h\chi_{E_-})^*(a-t)\,d\varphi(t)+\varphi(a-m(E_-))\nonumber\\ &=& 1+\int_0^{m(E_+)}(h\chi_{E_+})^*(t)\,d\varphi(t)+\int_0^{m(E_-)}(h\chi_{E_-})^*(t)\,d\varphi(a-t).
\label{equ3a}
\end{eqnarray}

Similarly, 
we have
\begin{equation}
\label{equ4a}
\|1-h\|_{\Lambda(\varphi)}=1+\int_0^{m(E_-)}(h\chi_{E_-})^*(t)\,d\varphi(t)+\int_0^{m(E_+)}(h\chi_{E_+})^*(t)\,d\varphi(a-t).
\end{equation}

Note that 
\begin{equation}
\label{equ4ab}
0<m(E_+)<a.
\end{equation}
Indeed, assume that $m(E_+)=0$ or $m(E_+)=a$. Then, one of the following conditions holds: (a) $m(E_+)=0$ and $m(E_-)>0$, (b) $m(E_-)=0$ and $m(E_+)>0$, (c) $m(E_+)=m(E_-)=0$.

In the case (a) we have 
$$
\int_0^{m(E_+)}(h\chi_{E_+})^*(t)\,d\varphi(t)=0$$
and
$$
\int_0^{m(E_-)}(h\chi_{E_-})^*(t)\,d\varphi(a-t)<0,$$
since $\varphi$ is strictly increasing and $h(t)<0$ if $t\in E_-$.
Consequently, by equation \eqref{equ3a}, we see that $\|1+h\|_{\Lambda(\varphi)}<1$.
But this impossible, because it contradicts condition \eqref{equ2a}.

In the case (b), arguing similarly, we obtain that $\|1+h\|_{\Lambda(\varphi)}>1$, which is impossible also in view of  \eqref{equ2a}.

Consider the case (c). The condition $m(E_+)=m(E_-)=0$ is equivalent to the fact that $h=0$ a.e. on $ [0,a]$. 
Then, $g=0$ a.e. on $[0,a]$, which contradicts the assumption. As a result, \eqref{equ4ab} is proved.
 


Next, in view of \eqref{equ4ab} and the definition of the sets $E_+$ and $E_-$, the functions  $\rho_1:=h\chi_{E_+}$ and $\rho_2:=h\chi_{E_-}$ satisfy the conditions of Lemma \ref{prop1a}. Therefore, we have 
\begin{equation}
\label{equ5a}
\int_0^{m(E_+)}(h\chi_{E_+})^*(t)\,d\varphi(t)>-\int_0^{m(E_+)}(h\chi_{E_+})^*(t)\,d\varphi(a-t).
\end{equation}
and
\begin{equation}
\label{equ6a}
\int_0^{m(E_-)}(h\chi_{E_-})^*(t)\,d\varphi(t)>-\int_0^{m(E_-)}(h\chi_{E_-})^*(t)\,d\varphi(a-t).
\end{equation}

Assume that
$$
\int_0^{m(E_+)}(h\chi_{E_+})^*(t)\,d\varphi(t)\ge \int_0^{m(E_-)}(h\chi_{E_-})^*(t)\,d\varphi(t).$$
Then from \eqref{equ6a} it follows
$$
\int_0^{m(E_+)}(h\chi_{E_+})^*(t)\,d\varphi(t)>-\int_0^{m(E_-)}(h\chi_{E_-})^*(t)\,d\varphi(a-t),$$
or
$$
\int_0^{m(E_+)}(h\chi_{E_+})^*(t)\,d\varphi(t)+\int_0^{m(E_-)}(h\chi_{E_-})^*(t)\,d\varphi(a-t)>0.$$
As a result, appealing to \eqref{equ3a}, we conclude that $\|1+h\|_{\Lambda(\varphi)}>1$, which contradicts \eqref{equ2a}.

In the case when
$$
\int_0^{m(E_+)}(h\chi_{E_+})^*(t)\,d\varphi(t)\le \int_0^{m(E_-)}(h\chi_{E_-})^*(t)\,d\varphi(t),$$
applying analogous arguments and \eqref{equ4a}, we conclude that $\|1-h\|_{\Lambda(\varphi)}>1$, which again contradicts \eqref{equ2a}. So, the proof is completed.

\end{proof}

\begin{corollary}\label{Cor0}
$L^1$ is a unique Lorentz space, whose unit ball has no extreme points. 
\end{corollary}

Since $\hf$ is isometric to a subspace of the space $\Lambda(\varphi)$, we get the following extension of the above-mentioned result of the paper \cite{CC} (see Section \ref{S0}) to the case when, in general, $\varphi$ fails to be strictly concave.

\begin{corollary}\label{Cor1}
Suppose $\varphi$ is a nonlinear, strictly increasing, concave, continuous function on $[0,2\pi]$ with $\varphi(0)=0$ and $\varphi(2\pi)=1$.  Then, each inner  function is an extreme point of the unit ball of ${\hf}$. 
\end{corollary}

Combining Corollary \ref{Cor1} with de Leeuw-Rudin theorem (see Section \ref{S0}) gives the following compliment to the latter theorem.

\begin{theorem}\label{Th3}
Suppose $\varphi$ is a strictly increasing, concave, continuous function on $[0,2\pi]$ such that $\varphi(0)=0$, $\varphi(2\pi)=1$.
The set of extreme points of the unit ball of the space ${\hf}$ coincides with the set of normed outer functions if and only if $\varphi(t)=t/(2\pi)$, i.e., ${\hf}=H^1$.
\end{theorem}

Let us now move on to the second main result of the paper.

Suppose $\varphi$ is a strictly increasing and strictly  concave  function on $[0,2\pi]$ with $\varphi(0)=0$ and $\varphi(2\pi)=1$. Then, according to the main result of the note \cite{CC}, every inner  function is an extreme point of the unit ball of ${\hf}$. However, as was explained already in Section \ref{S0}, this is really a very easy consequence of Theorem \ref{t-0} and the facts that the function $\tilde{f}(t):=f(e^{it})$, $0\le t\le 2\pi$, is an extreme point of the unit ball of the space $\Lambda(\varphi)$ and mapping \eqref{map} 
includes ${\hf}$ into $\Lambda(\varphi)$ as an isometric subspace.


In the following theorem we show that the same assertion holds (under the same conditions imposed on $\varphi$) for each function $f\in\hf$, $\|f\|_{\hf}=1$, satisfying the much weaker condition that the function $|\tilde{f}|$ is a constant on some set of positive measure. Note that such a function needs to be neither inner, nor outer. 



\begin{theorem}\label{Th2a}
Let $\varphi$ be a strictly increasing and strictly  concave  function on $[0,2\pi]$ such that $\varphi(0)=0$ and $\varphi(2\pi)=1$. Assume that  $f\in\hf$, $\|f\|_{\hf}=1$. Moreover, let the function $\mu:=|\tilde{f}|$ be a constant on some subset of positive measure  of the interval $[0,2\pi]$. 

Then, $f$ is an extreme point of the unit ball of the space $\hf$.
\end{theorem}
\begin{proof}
By the assumption, we have that $m\{s\in [0,2\pi]:\,\mu(s)=c\}>0$ for some $c\ge 0$.

To the contrary, assume that $f$ fails to be an extreme point of the unit ball of $\hf$. Thus, there exists a function $g\in \hf$, $g\ne 0$,  for which it holds
\begin{equation}\label{e1}
\|\tilde{f}\pm \tilde{g}\|_{\Lambda(\varphi)}=\|f\pm g\|_{\hf}=1.
\end{equation}

Since $\varphi$ is being strictly increasing, the norm of the space $\Lambda(\varphi)$ is strictly monotone (see Lemma \ref{0}). Moreover, Lusin-Privalov uniqueness theorem (see, e.g., \cite[\S\,III.$3^0$]{koosis}) combined with the condition $\|f\|_{\hf}=1$ implies that $m\{t:\,\tilde{f}(t)=0\}=0$. Therefore, by Lemma \ref{l0}, we have $\tilde{g}=\tilde{f}\cdot h$, where $h$ is real-valued and $|h(t)|\le 1$ a.e. on $[0,2\pi]$.  Hence, both functions $\mu\pm \mu h=\mu(1\pm h)$ are nonnegative, and from \eqref{e1} it follows that 
\begin{equation}\label{e1dop}
\|\mu(1\pm h)\|_{\Lambda(\varphi)}=1.
\end{equation}

Next, according to Proposition \ref{propos}, there exists a family of subsets $\{E_t\}_{0<t\le 2\pi}$ of the interval $[0,2\pi]$ such that 
\begin{equation}
\label{pr1mu}
\int_{E_t}\mu(s)\,ds=\int_0^t \mu^*(s)\,ds,\;\;0<t\le 2\pi,
\end{equation}
and
\begin{equation}
\label{pr2mu}
\int_{E_t}\mu(s)(1\pm h(s))\,ds=\int_0^t (\mu(1\pm h))^*(s)\,ds,\;\;0<t\le 2\pi.
\end{equation}
Combining \eqref{pr2mu} with \cite[\S\,II.2, Property $8^0$]{krein-petunin-semenov}), for both plus and minus signs we get:
\begin{equation}
\label{pr4mu-e}
\int_{E_t}\mu(s)(1\pm h(s))\,ds=\sup\Big\{\int_{A}\mu(s)(1\pm h(s))\,ds:\,A\subset [0,2\pi], m(A)=t\Big\}.
\end{equation}

Let us denote
$$
C:=\{s\in [0,2\pi]:\,\mu(s)=c\},\;\;C_>:=\{s\in [0,2\pi]:\,\mu(s)>c\}.$$
Since $\mu^*$ and $\mu$ are equimeasurable and the rearrangement $\mu^*$ decreases, it follows (up to a set of zero measure) that
$$
\{s\in [0,2\pi]:\,\mu^*(s)=c\}=(a,b),$$
where
$$
a=m(C_>)=m\{s\in [0,2\pi]:\,\mu^*(s)>c\},\;b=m(C)+a.
$$


Let $t\in (a,b)$ be arbitrary. Since $\mu^*(t)=c$, by the definition of the sets $C$ and $C_>$, we have  
$$
\{s\in [0,2\pi]:\,\mu(s)=\mu^*(t)\}=C,\;\;\{s\in [0,2\pi]:\,\mu(s)>\mu^*(t)\}=C_>.$$
Consequently, from the fact that $m(E_t)=t$ and the embeddings
$$
\{s\in [0,2\pi]:\,\mu(s)>\mu^*(t)\}\subset E_t\subset \{s\in [0,2\pi]:\,\mu(s)\ge\mu^*(t)\}$$
(see \eqref{pra1}) it follows that
\begin{equation}
\label{pr1mu-e}
E_t=C_>\cup F_t,
\end{equation}
where $F_t\subset C$ and $m(F_t)=t-a$ (in the case when $a=0$ we have $C_>=\emptyset$).
Thus, since $\mu(s)=c$ if $s\in C$, for all $t\in (a,b)$ it holds
\begin{eqnarray}
\int_{E_t}\mu(s)(1\pm h(s))\,ds&=&\int_{C_>}\mu(s)(1\pm h(s))\,ds+\int_{F_t}\mu(s)(1\pm h(s))\,ds\nonumber\\
&=&\int_{C_>}\mu(s)(1\pm h(s))\,ds+c\Big(t-a\pm \int_{F_t}h(s)\,ds\Big).
\label{pr2mu-e}
\end{eqnarray}

We claim that the function $h$ is constant a.e. on the set $C$, i.e., 
\begin{equation}\label{5}
h(t)=c_0\;\;\mbox{a.e. on}\; C\;\;\mbox{for some}\;c_0\in [-1,1].
\end{equation}

Assuming that it is not the case, by continuity of the Lebesgue measure, we can find
$t_0\in (a,b)$ and two subsets $F'$ and $F''$ of $C$
such that $m(F')=m(F'')=t_0-a$ and 
\begin{equation}
\label{pr3mu-e}
\alpha:=\sup_{s\in F''}\,h(s)<\beta:=\inf_{s\in F'}\,h(s).
\end{equation}
Let us fix such a $t_0$.

Setting $E':=C_>\cup F'$ and $E'':=C_>\cup F''$, as above, we obtain
\begin{equation}
\int_{E'}\mu(s)(1\pm h(s))\,ds=\int_{C_>}\mu(s)(1\pm h(s))\,ds+c\Big(t_0-a\pm \int_{F'}h(s)\,ds\Big)
\label{prnew1}
\end{equation}
and
\begin{equation}
\int_{E''}\mu(s)(1\pm h(s))\,ds=\int_{C_>}\mu(s)(1\pm h(s))\,ds+c\Big(t_0-a\pm \int_{F''}h(s)\,ds\Big).
\label{prnew2}
\end{equation}
Since $m(E')=t_0$, from \eqref{pr4mu-e} it follows
$$
\int_{E_{t_0}}\mu(s)(1+h(s))\,ds\ge \int_{E'}\mu(s)(1+h(s))\,ds.$$
Hence, comparing the right-hand sides of equations \eqref{pr2mu-e} (with $t=t_0$) and \eqref{prnew1} for the plus sign and applying \eqref{pr3mu-e}, we obtain
$$
\int_{F_{t_0}}h(s)\,ds\ge \int_{F'}h(s)\,ds\ge\beta(t_0-a)>\alpha(t_0-a)\ge \int_{F''}h(s)\,ds.$$
Consequently, by \eqref{pr2mu-e} (with $t=t_0$) and \eqref{prnew2} for the minus sign,
\begin{eqnarray*}
\int_{E_{t_0}}\mu(s)(1-h(s))\,ds&<&\int_{C_>}\mu(s)(1-h(s))\,ds+c\Big(t_0-a-\int_{F''}h(s)\,ds\Big)\\&=&\int_{E''}\mu(s)(1-h(s))\,ds.
\end{eqnarray*}
Since $m(E'')=t_0$, the last inequality contradicts \eqref{pr4mu-e} for $t=t_0$ and the minus sign.
As a result, claim \eqref{5} is proved.

In view of \eqref{5} and the definition of $h$, we have that $
\tilde{g}=c_0\cdot \tilde{f}$ a.e. on $C$.
This implies that the nontangential limit of the analytic function $g(z)-c_0 f(z)$ as $|z|\to 1$ vanishes on some subset of positive measure of the circle $\mathbb{T}$. Hence, applying Lusin-Privalov uniqueness theorem (see e.g. \cite[\S\,III.$3^0$]{koosis}) once more, we conclude that $g(z)=c_0 f(z)$ for all $z\in\D$. Hence, $h=c_0$ a.e. on $[0,2\pi]$. In view of  \eqref{e1dop}, this implies that
$$
(1+c_0)\|f\|_{\hf}=(1-c_0)\|f\|_{\hf}=1,$$
whence $c_0=0$ or equivalently $g=0$, which contradicts the assumption. This completes the proof.

\end{proof}



\end{document}